\documentclass[a4paper,11pt]{amsart}

\usepackage{color}
\usepackage{textcomp}

\usepackage{hyperref}
\usepackage[capitalise]{cleveref}

\usepackage{amsfonts,amssymb,mathtools}
\usepackage{enumerate}
\usepackage{wasysym}

\theoremstyle{plain}

\newtheorem*{main}{Main Theorem}
\newtheorem{thm}{Theorem}[section]
\newtheorem{lem}[thm]{Lemma}
\newtheorem{pro}[thm]{Proposition}

\theoremstyle{definition}

\newtheorem{dfn}[thm]{Definition}

\newtheorem{rmk}[thm]{Remark}

\newcommand{\F}{\mathbb{F}}
\newcommand{\Z}{\mathbb{Z}}

\newcommand{\N}{\mathbb{N}}

\newcommand{\C}{\mathbb{C}}
\renewcommand{\P}{\mathbb{P}}

\DeclareMathOperator{\rank}{rank}
\DeclareMathOperator{\bea}{b}
\DeclareMathOperator{\bq}{bq}
\DeclareMathOperator{\nbea}{nb}

\begin{document}

\title[On the number of Beauville and non-Beauville $p$-groups]{On the asymptotic behaviour of the number of Beauville and non-Beauville $p$-groups}

\author[G.A.\ Fern\'andez-Alcober]{Gustavo A.\ Fern\'andez-Alcober}
\address{Department of Mathematics\\ University of the Basque Country UPV/EHU\\
48080 Bilbao, Spain}
\email{gustavo.fernandez@ehu.eus}

\author[\c{S}.\ G\"ul]{\c{S}\"ukran G\"ul}
\address{Department of Mathematics\\ University of the Basque Country UPV/EHU\\
48080 Bilbao, Spain}
\email{sukran.gul@ehu.eus}

\author[M.\ Vannacci]{Matteo Vannacci}
\address{Mathematisches Institut der Heinrich-Heine-Universit\"at, Universit\"atsstr.~1, 40225 D\"usseldorf, Germany}
\email{vannacci@uni-duesseldorf.de}

\keywords{Beauville groups; finite $p$-groups\vspace{3pt}}


\thanks{The first two authors are supported by the Spanish Government, grants
MTM2014-53810-C2-2-P and MTM2017-86802-P, and by the Basque Government, grant IT974-16.
The third author acknowledges support from the research training group
\textit{GRK 2240: Algebro-Geometric Methods in Algebra, Arithmetic and Topology}, funded by the DFG}

\begin{abstract}
We find asymptotic lower bounds for the numbers of both Beauville and non-Beauville $2$-generator finite $p$-groups of a fixed order, which turn out to coincide with the best known asymptotic lower bound for the total number of $2$-generator finite $p$-groups of the same order.
This shows that both Beauville and non-Beauville groups are abundant within the family of finite $p$-groups.
\end{abstract}

\maketitle

\section{Introduction}

A \emph{Beauville surface\/} (of unmixed type) is a compact complex surface isomorphic to a quotient
$(C_1\times C_2)/G$ satisfying the following conditions:
\begin{enumerate}
\item
$C_1$ and $C_2$ are algebraic curves of genera at least $2$.
\item
$G$ is a finite group acting freely on $C_1\times C_2$ by holomorphic transformations.
\item
For both $i=1,2$, we have $C_i/G\cong \P_1(\C)$ and the covering map $C_i\rightarrow C_i/G$ is ramified over three points.
\end{enumerate}
Then the group $G$ is said to be a \emph{Beauville group\/}.

The geometric definition of Beauville groups can be transformed into purely group-theoretical terms.
Given a group $G$, a tuple $T=(x,y,z)$ is called a \emph{spherical triple of generators\/} if $G=\langle x,y,z \rangle$ and $xyz=1$.
Then $G$ is a Beauville group if and only if it admits two spherical triples of generators, $T_1$ and $T_2$, such that
\begin{equation}
\label{beauville condition}
\langle t_1 \rangle^g \cap \langle t_2 \rangle = 1,
\quad
\text{for all $t_1\in T_1$, $t_2\in T_2$, and $g\in G$.}
\end{equation}
(With a slight abuse of notation, we use the same symbol for the tuple and the set of its components, and we write the $\in$ symbol to denote that an element appears in the tuple.)
We then say that $(T_1,T_2)$ is a \emph{Beauville structure\/} of $G$.
Thus Beauville groups are in particular $2$-generator finite groups.

At this point it is natural to ask which $2$-generator finite groups are Beauville groups.
In 2000, Catanese \cite{cat} proved that a finite abelian group is a Beauville group if and only if it is isomorphic to $C_n\times C_n$, where $n>1$ and $\gcd(n,6)=1$.
On the other hand, all finite simple groups other than $A_5$ are Beauville groups, as shown by Guralnick and Malle in 2012 \cite{GM}.
Fairbairn, Magaard and Parker \cite{FMP,FMPc} proved more generally that all finite quasisimple groups are Beauville groups, with the only exceptions of $A_5$ and $SL_2(5)$.

The determination of nilpotent Beauville groups is easily reduced to the case of $p$-groups.
In \cite{FG}, the first two authors found a criterion for a finite $p$-group with a `nice power structure' to be a Beauville group that extends Catanese's condition for abelian groups.
More precisely, if $G$ is a $2$-generator finite $p$-group of exponent $p^e$ that satisfies the condition
\begin{equation}
\label{nice power structure}
x^{p^{e-1}}=y^{p^{e-1}} \quad \Longleftrightarrow \quad (xy^{-1})^{p^{e-1}}=1,
\end{equation}
then $G$ is a Beauville group if and only if $p\ge 5$ and $|G^{p^{e-1}}|\ge p^2$.
Note that condition \eqref{nice power structure} holds for usual families of $p$-groups such as powerful groups, $p$-central groups, and regular $p$-groups.
In particular, it holds for $p$-groups of nilpotency class less than $p$.
On the other hand, examples are given in \cite{FG} showing that the above criterion is not generally valid without assumption \eqref{nice power structure}, and it may be very hard to decide whether a given finite $p$-group has a Beauville structure.

Failing to find a general recipe to tell apart Beauville $p$-groups from non-Beauville ones, we can reformulate our goal and try to determine the asymptotic behaviour of the number of Beauville groups of order $p^t$ as $t$ tends to infinity.
In this respect, observe that Jaikin-Zapirain \cite{J} provided asymptotic bounds for the number
$f_d(p^t)$ of $d$-generator finite $p$-groups of order $p^t$, which reduce to the following in the case $d=2$:
\begin{equation}
\label{jaikin bounds}
p^{\frac{1}{4}t^2+o(t^2)} \le f_2(p^t) \le p^{\frac{1}{2}t^2+o(t^2)}.
\end{equation}
In \cite{FGJ}, the authors comment that it is very plausible that most $2$-generator finite $p$-groups of sufficiently large order are Beauville groups.
In this paper, we shed some light on this problem, by showing that there are an abundance of
\emph{both\/} Beauville and non-Beauville groups among $2$-generator finite $p$-groups.
More precisely, if $\mathrm{b}(p^t)$ and $\mathrm{nb}(p^t)$ denote the number of Beauville and non-Beauville $2$-generator $p$-groups of order $p^t$, we show that these numbers are asymptotically at least as big as the best known lower bound for the total number of $2$-generator groups of order $p^t$, provided that $p\ge 5$.

\begin{main}
Let $p\ge 5$ be a prime.
Then
\[
\mathrm{b}(p^t) \ge p^{\frac{1}{4}t^2+o(t^2)}
\qquad
\text{and}
\qquad
\mathrm{nb}(p^t) \ge p^{\frac{1}{4}t^2+o(t^2)}.
\]
\end{main}

The proof of this theorem follows ideas used by Jaikin-Zapirain in order to obtain the lower bound for
$f_2(p^t)$ in \eqref{jaikin bounds}.
Roughly speaking, the approach is to consider the free group $F$ on $2$ generators and a normal subgroup $H$ of $p$-power index in $F$.
Then one refines the lower $p$-central series of $H$ to a special type of normal series of $F$ with all sections of order $p$, which we call \emph{standard series\/}.
Now for every $S$ in a standard series, let us consider all quotients $F/W$ with $S^p[S,F]\le W\le S$.
We refer to these factor groups as \emph{standard $2$-generator $p$-groups\/}.
Then Jaikin-Zapirain's argument is to observe that, for a given $n\in\N$, by suitably choosing $S$ and $W$, one can find as many non-isomorphic standard groups as shown in the lower bound of
\eqref{jaikin bounds}.

In counting Beauville and non-Beauville $2$-generator $p$-groups, again we use standard $p$-groups, and we need to perform a very careful analysis of the power-commutator structure of such groups.
This is the goal of Section 2, where we work with the lower $p$-central series of a general group $G$, not necessarily free.
We also introduce standard series and another refinement of the lower $p$-central series, that we will write as $\{Q_{n,k}(G)\}$, and that will play a significant role in the determination of the existence of Beauville structures.
In Section 3, we consider the above mentioned situation of a free group $F$ on $2$ generators and a normal subgroup $H$ of $p$-power index, and we are able to give a criterion to determine whether the quotients of $F$ by a quite general family of normal subgroups are Beauville groups.
This applies in particular to all standard groups arising from a subgroup $S$ that is not lying strictly between $P_{n+1}(H)$ and $Q_{n,1}(H)$.
This criterion relies heavily on the theory developed in Section 2, but it also requires some other results that are specific to free groups.
The restriction on $S$ only means a bounded number of exceptions for every $n$, and it allows us to adjust Jaikin-Zapirain's argument in order to count Beauville and non-Beauville groups, which we do in Section 4.

The restriction to $p\ge 5$ in our main theorem does not come as a surprise.
One of the cornerstones in our method for finding Beauville structures in Section 3 is the fact that the lower $p$-central quotients $F/P_n(H)$ are always Beauville groups for $p\ge 5$.
However, as the second author showed in \cite{gul}, $F/P_n(F)$ is never a Beauville group if $p=2$ or $3$, a result that can be extended with a similar argument to the quotients $F/P_n(H)$.
Also, in general, the construction of Beauville $p$-groups is much more difficult when $p=2$ or $3$ than for $p\ge 5$.
Actually, for some time it was not known whether there existed any Beauville $2$-groups or $3$-groups, and the first examples were given in \cite[Section 5]{FGJ}.
It is certainly an interesting problem to try to find a lower bound for $\mathrm{b}(p^t)$ when $p=2$ or $3$, but it will require a radically different approach.

\section{Power-commutator calculus and refinements of the lower $p$-central series}

Let $G$ be a group and let $p$ be a prime number.
The series
\[
G=P_1(G) \geq P_2(G)\geq \dots \geq P_n(G)\geq \dots
\]
given by $P_n(G)=[P_{n-1}(G),G]P_{n-1}(G)^p$ for $n>1$ is called the \emph{lower $p$-central series\/} of $G$.
By Definition 1.4 and Theorem 1.5 in \cite{HB}, we have
\[
P_n(G)= \gamma_1(G)^{p^{n-1}}\gamma_2(G)^{p^{n-2}}\dots \gamma_n(G).
\]
We refer to each quotient group $G/P_n(G)$ as a \emph{lower $p$-central quotient\/} of $G$.
Observe that, if $G$ is finitely generated, then $G/P_n(G)$ is a finite $p$-group of exponent less than $p^n$ and class less than $n$.
We define a function $s_G\colon\N\rightarrow\N$ by means of
\[
|P_n(G):P_{n+1}(G)| = p^{s_G(n)}.
\]
If $F$ is the free group on $d$ generators, we write $s_d(n)$ for $s_F(n)$.
By Lemma 3.1 of \cite{J}, we can write
\[
s_d(n) = \sum_{i=1}^n \, M_d(i),
\]
with
\[
M_d(n) = \frac{1}{n} \sum_{k\mid n} \, \mu(k) d^{n/k},
\]
where $\mu$ is the M\"obius function.
Also, one can easily check that $s_d(n)\le (d+1)s_d(n-1)$, which in turn yields that
\[
s_d(n) \le d \sum_{i=1}^{n-1} \, s_d(i),
\]
or what is the same, that
\begin{equation}
\label{Pn/Pn+1 versus F/Pn}
|P_n(F):P_{n+1}(F)|\le |F:P_n(F)|^d.
\end{equation}

Our determination of both Beauville and non-Beauville groups among standard $2$-generator $p$-groups requires introducing some refinements of the lower $p$-central series.
To this purpose, we first develop a collection of tools regarding power-commutator calculus with respect to the terms of the lower $p$-central series.

Recall that by the Hall-Petrescu identity (see Lemma 1.1 in Chapter VIII of \cite{HB}), given a prime $p$ and $m\in\N\cup\{0\}$, for any two elements $g$ and $h$ of a group we have
\begin{equation}
\label{hall-petrescu}
(gh)^{p^m} \equiv g^{p^m} h^{p^m}
\pmod{\gamma_2(H)^{p^m} \prod_{r=1}^m \gamma_{p^r}(H)^{p^{m-r}}},
\end{equation}
where $H=\langle g,h \rangle$, and
\begin{equation}
\label{hall-petrescu comm}
[g^{p^m},h] \equiv [g,h]^{p^m}
\pmod{\gamma_2(K)^{p^m} \prod_{r=1}^m \gamma_{p^r}(K)^{p^{m-r}}},
\end{equation}
where $K=\langle g, [g,h] \rangle$.

We also need the following fact about commutators and powers of the subgroups $P_n(G)$, which is Theorem 1.5 in Chapter VIII of \cite{HB}.

\begin{lem}
\label{BH lemma}
Let $G$ be a group.
Then, for all $m,n\in\N$, both $[P_n(G),P_m(G)]$ and $P_n(G)^{p^m}$ are contained in $P_{n+m}(G)$.
\end{lem}

Now we can translate \eqref{hall-petrescu} and \eqref{hall-petrescu comm} into congruences with respect to subgroups of the lower $p$-central series.

\begin{pro}
\label{hall-petrescu modulo Q}
Let $G$ be a group.
Then, for every odd prime $p$ and every $m\in\N\cup\{0\}$ and $n\in\N$, the following hold:
\begin{enumerate}
\item
For every $a,b\in P_n(G)$, we have
\begin{equation}
\label{power modulo P}
(ab)^{p^m} \equiv a^{p^m} b^{p^m} \pmod{P_{2n+m}(G)}.
\end{equation}
If $p\ge 5$ and $a,b\in\gamma_n(G)$ then we furthermore have
\begin{equation}
\label{power modulo P p>3}
(ab)^{p^m} \equiv a^{p^m} b^{p^m}
\pmod{\gamma_{2n}(G)^{p^m} P_{2n+m+1}(G)}.
\end{equation}
\item
For every $a\in P_n(G)$ and $g\in G$, we have
\begin{equation}
\label{power-comm modulo P}
[a^{p^m},g] \equiv [a,g]^{p^m} \pmod{P_{2n+m+1}(G)}.
\end{equation}
\end{enumerate}
\end{pro}

\begin{proof}
(i)
We apply the Hall-Petrescu identity to $(ab)^{p^m}$ with $a,b\in P_n(G)$.
Set $H=\langle a,b \rangle$.
Then, by \cref{BH lemma}, the modulus in \eqref{hall-petrescu} is contained in
\[
P_{2n+m}(G) \prod_{r=1}^m P_{p^r n+m-r}(G)
\le 
P_{2n+m}(G) P_{pn+m-1}(G)
\le
P_{2n+m}(G),
\]
since $p$ is odd.
If $p\ge 5$ and $a,b\in\gamma_n(G)$ then the modulus is contained in
\[
\gamma_{2n}(G)^{p^m} P_{pn+m-1}(G) \le \gamma_{2n}(G)^{p^m} P_{2n+m+1}(G).
\]
The result follows.

(ii)
Now we apply the commutator Hall-Petrescu identity \eqref{hall-petrescu comm} to $[a^{p^m},g]$ with $a\in P_n(G)$ and $g\in G$.
We only have to observe that, if $K=\langle a, [a,g] \rangle$, then
\[
\gamma_2(K)^{p^m}\le [P_n(G),G,P_n(G)]^{p^m}\le P_{2n+1}(G)^{p^m} \le P_{2n+m+1}(G),
\]
and for all $r=1,\ldots,m$,
\[
\gamma_{p^r}(K)^{p^{m-r}} \le P_{p^rn+m-r+1}(G) \le P_{pn+m}(G) \le P_{2n+m+1}(G),
\]
since $p$ is odd.
\end{proof}

\begin{pro}
\label{powers and comms mod Pn+2}
Let $G$ be a group, let $p$ be an odd prime, and let $m\in\N\cup\{0\}$ and $n\in\N$.
Then:
\begin{enumerate}
\item
The map
\[
\begin{matrix}
P_n(G) & \longrightarrow & P_{n+m}(G)/P_{n+m+1}(G)
\\[5pt]
x & \longmapsto & x^{p^m}P_{n+m+1}(G)
\end{matrix}
\]
is a group homomorphism.
\end{enumerate}
Furthermore, if $L$ and $N$ are normal subgroups of $G$ contained in $P_n(G)$ then the following hold:
\vspace{3pt}
\begin{enumerate}
\setcounter{enumi}{1}
\setlength\itemsep{7pt}
\item
The subgroup $L^{p^m}$ and the set $\{x^{p^m}\mid x\in L\}$ coincide modulo
$P_{n+m+1}(G)$.
\item
We have
\begin{equation}
\label{exp of product}
(LN)^{p^m} P_{n+m+1}(G) = L^{p^m} N^{p^m} P_{n+m+1}(G).
\end{equation}
\item
For every $s\in\{0,\ldots,m\}$,
\begin{equation}
\label{double exp}
(L^{p^s})^{p^{m-s}} P_{n+m+1}(G) = L^{p^m} P_{n+m+1}(G).
\end{equation}
\item
We have
\begin{equation}
\label{exp inside comm}
[L,G]^{p^m} P_{n+m+2}(G) = [L^{p^m},G] P_{n+m+2}(G).
\end{equation}
\end{enumerate}
\end{pro}

\begin{proof}
Note that (i) is an immediate consequence of \eqref{power modulo P}, and then (ii) and (iii) readily follow.
As for (iv), the inclusion $\supseteq$ is trivial.
On the other hand, if $y\in L^{p^s}$ then, by (ii), $y=x^{p^s}z$ for some $x\in L$ and $z\in P_{n+s+1}(G)$, and then 
\[
y^{p^{m-s}} \equiv x^{p^m} z^{p^{m-s}} \pmod{P_{n+m+1}(G)}.
\]
Since $P_{n+s+1}(G)^{p^{m-s}}\le P_{n+m+1}(G)$ by \cref{BH lemma}, we get the other inclusion.
Finally, (v) follows from \eqref{power-comm modulo P}, by using (i) and that
$[P_{n+m+1}(G),G]\le P_{n+m+2}(G)$.
\end{proof}

Since $P_{n+1}(G)^{p^m}\le P_{n+m+1}(G)$, it follows from (i) in the previous proposition that the $p^m$th power map induces a homomorphism
\[
\pi_{n,m} \colon P_n(G)/P_{n+1}(G)\rightarrow P_{n+m}(G)/P_{n+m+1}(G).
\]
We next see that this is a monomorphism if $G$ is a free group.
We need the following result, which is Theorem~1.8(a) in Chapter VIII of \cite{HB}.

\begin{thm}
\label{huppert bijection}
Let $F$ be a free group and let $p$ be an odd prime.
For every $n\in\N$, set
\[
\mathfrak{D}_n(F)
=
\frac{F}{F^p \gamma_2(F)} \times \frac{\gamma_2(F)}{\gamma_2(F)^p \gamma_3(F)} \times \cdots \times \frac{\gamma_n(F)}{\gamma_n(F)^p \gamma_{n+1}(F)}.
\]
Then the map
\[
\begin{matrix}
f_n & \colon & \mathfrak{D}_n(F)  & \longrightarrow & P_n(F)/P_{n+1}(F)
\\[5pt]
& & (\overline g_1,\ldots,\overline g_n) &  \longmapsto 
& g_1^{p^{n-1}} g_2^{p^{n-2}} \ldots g_{n-1}^p g_n P_{n+1}(F)
\end{matrix}
\]
is an isomorphism.
(Here, $\overline g_i$ denotes the image of $g_i\in\gamma_i(F)$ modulo
$\gamma_i(F)^p\gamma_{i+1}(F)$ for every $i=1,\ldots,n$.)
\end{thm}

\begin{lem}
\label{isomorphism free case}
If $F$ is a free group, then the map
\[
\begin{matrix}
\pi_{n,m} & \colon & P_n(F)/P_{n+1}(F) & \longrightarrow & P_{n+m}(F)/P_{n+m+1}(F)
\\[5pt]
& & xP_{n+1}(F) & \longmapsto & x^{p^m}P_{n+m+1}(F)
\end{matrix}
\]
is a monomorphism.
\end{lem}

\begin{proof}
From the definition of the map $f_{n+m}$, it follows that
\[
| \pi_{n,m}(P_n(F)/P_{n+1}(F)) |
\ge
\prod_{i=1}^n \, |\gamma_i(F)/\gamma_i(F)^p\gamma_{i+1}(F)|,
\]
and since $f_n$ is bijection, the latter coincides with $|P_n(F)/P_{n+1}(F)|$.
This implies that $\pi_{n,m}$ is injective.
\end{proof}

Now we introduce a first refinement of the lower $p$-central series.

\begin{dfn}
Let $G$ be a group, let $p$ be a prime, and let $n$ and $k$ be positive integers, with $k\le n$.
Then we define the following subgroup of $P_n(G)$:
\[
Q_{n,k}(G) = G^{p^{n-1}} \gamma_2(G)^{p^{n-2}} \ldots \gamma_{k-1}(G)^{p^{n-k+1}} \gamma_k(G)^{p^{n-k}} P_{n+1}(G).
\]
\end{dfn}

\vspace{10pt}

Thus we have $P_n(G)=Q_{n,n}(G)$,  and if we glue together the series $\{Q_{n,k}(G)\}_{k=1}^n$, we get a refinement of the lower $p$-central series of $G$ consisting of verbal subgroups, so in particular of characteristic subgroups of $G$.
Observe that the first two terms of this refinement between $P_{n+1}(G)$ and $P_n(G)$ are
\[
Q_{n,1}(G) = G^{p^{n-1}} P_{n+1}(G)
\quad
\text{and}
\quad
Q_{n,2}(G) = G^{p^{n-1}} \gamma_2(G)^{p^{n-2}} P_{n+1}(G).
\]

\begin{pro}
\label{special subgroups}
Let $G$ be a group and let $p$ be an odd prime.
Then, for every $1\le k \le n$, we have
\[
Q_{n,k}(G)^p [Q_{n,k}(G), G] = Q_{n+1, k+1}(G).
\]
\end{pro}

\begin{proof}
By using different items of \cref{powers and comms mod Pn+2}, we have
\begin{equation}
\label{Q power}
\begin{split}
Q_{n,k}(G)^p P_{n+2}(G)
&=
\Big( \prod_{r=1}^k \, (\gamma_r(G)^{p^{n-r}})^p \Big) P_{n+1}(G)^p P_{n+2}(G)
\\
&=
\Big( \prod_{r=1}^k \, \gamma_r(G)^{p^{n+1-r}} \Big) P_{n+2}(G)
=
Q_{n+1,k}(G),
\end{split}
\end{equation}
and
\begin{equation}
\label{Q commutator}
\begin{split}
[Q_{n,k}(G),G] P_{n+2}(G)
&= \Big( \prod_{r=1}^k \, [\gamma_r(G)^{p^{n-r}},G] \Big) [P_{n+1}(G),G]  P_{n+2}(G)
\\
&= \Big( \prod_{r=1}^k \, \gamma_{r+1}(G)^{p^{n-r}} \Big) P_{n+2}(G).
\end{split}
\end{equation}
By multiplying \eqref{Q power} and \eqref{Q commutator} together, we get
\[
Q_{n,k}(G)^p[Q_{n,k}(G),G] P_{n+2}(G)=Q_{n+1,k+1}(G).
\]
Now, since $P_{n+1}(G)$ is contained in $Q_{n,k}(G)$,
\[
P_{n+2}(G)=P_{n+1}(G)^p[P_{n+1}(G),G] \le Q_{n,k}(G)^p [Q_{n,k}(G), G]
\]
and the result follows.
\end{proof}

We will also need to know the behaviour of $p^{n-2}$nd powers with respect to the subgroup
$Q_{n,2}(G)$.

\begin{pro}
\label{pn-2 powers}
Let $G$ be a group and let $g,h\in G$.
If $p\ge 5$ is a prime and $n\ge 2$ is an integer then
\begin{equation}
\label{pn-2 powers-1}
(gh)^{p^{n-2}} \equiv g^{p^{n-2}} h^{p^{n-2}} \pmod{Q_{n,2}(G)},
\end{equation}
and if $g\equiv h \pmod{P_2(G)}$ then
\begin{equation}
\label{pn-2 powers-2}
g^{p^{n-2}} \equiv h^{p^{n-2}} \pmod{Q_{n,2}(G)}.
\end{equation}
\end{pro}

\begin{proof}
By \eqref{power modulo P p>3}, we have
\[
(gh)^{p^{n-2}} \equiv g^{p^{n-2}} h^{p^{n-2}} \pmod{\gamma_2(G)^{p^{n-2}} P_{n+1}(G)}.
\]
Since the modulus of this congruence is contained in $Q_{n,2}(G)$, this proves \eqref{pn-2 powers-1}.

On the other hand, we have
\[
P_2(G)^{p^{n-2}}
\le
(G^p)^{p^{n-2}} \gamma_2(G)^{p^{n-2}} Q_{n,2}(G)
=
Q_{n,2}(G),
\]
where the inclusion follows from \eqref{pn-2 powers-1}, and the equality from \eqref{double exp}.
Now if $g\equiv h \pmod{P_2(G)}$ then write $g=hz$ with $z\in P_2(G)$, and by applying
\eqref{pn-2 powers-1} and the last inclusion, we obtain \eqref{pn-2 powers-2}.
\end{proof}

In the case of free groups, more precise information can be given.

\begin{pro}
\label{linearly independence}
Let $F=\langle x_1, \dots, x_d \rangle$ be a free group of rank $d$ and let $p$ be an odd prime. Suppose that $N$ is a normal subgroup of $F$ such that $Q_{n,2}(F) \le N \le P_n(F)$ with $n\ge 2$. Then the elements $x_1^{p^{n-2}}, \dots, x_d^{p^{n-2}}$ are linearly independent modulo $N$.
\end{pro}

\begin{proof}
By \eqref{double exp} and \eqref{exp inside comm}, we have
\[
(F^{p^{n-2}})^p \le F^{p^{n-1}} P_{n+1}(F) \le Q_{n,2}(F)
\]
and
\[
[F^{p^{n-2}},F]\le \gamma_2(F)^{p^{n-2}} P_{n+1}(F) \le Q_{n,2}(F).
\]
\vspace{-5pt}

\noindent
Since $Q_{n,2}(F)\le N$, it follows that $F^{p^{n-2}}N/N$ is a central section of $F$ of exponent $p$.
In particular, it is an $\F_p$-vector space and it makes sense to speak of linear independence of
$p^{n-2}$nd powers modulo $N$.

Suppose, by way of contradiction, that $x_1^{p^{n-2}}, \dots, x_d^{p^{n-2}}$ are linearly dependent modulo $N$.
We may assume, without loss of generality, that
\[
x_1^{p^{n-2}} x_2^{i_2p^{n-2}} \ldots x_d^{i_dp^{n-2}} \in N
\]
for  some $i_2,\ldots,i_d \in \{0, \dots, p-1\}$.
Now we have $N\le P_n(F)$ and every element of $P_n(F)$ can be written in the form $g_1^{p^{n-1}} g_2^{p^{n-2}} \dots g_n$ for suitable $g_i \in \gamma_i(F)$, by \cref{huppert bijection}.
It follows that
\[
x_1^{p^{n-2}} x_2^{i_2p^{n-2}} \ldots x_d^{i_dp^{n-2}} = a^{-p^{n-1}}b
\]
for some $a \in F$ and $b\in \gamma_2(F)$.
Then we have 
\begin{equation}
\label{element in gamma2}
a^{p^{n-1}} x_1^{p^{n-2}} x_2^{i_2p^{n-2}} \ldots x_d^{i_dp^{n-2}}\in \gamma_2(F).
\end{equation}
Write $a= x_1^{j_1}\dots x_d^{j_d}z$ for some $j_1,\ldots,j_d\in \Z$ and $z\in \gamma_2(F)$.
Then the exponent sum of $x_1$ in (\ref{element in gamma2}) will be $p^{n-2}(1+j_1p)$.
On the other hand, an element of the free group $F$ belongs to $\gamma_2(F)$ if and only if  the exponent sum of each free generator is zero.
Hence we get  $p^{n-2}(1+j_1p)=0$, which is a contradiction.
\end{proof}

Next we introduce a second way of refining the lower $p$-central series, which is also a refinement of the series $\{Q_{n,k}(G)\}$.
In this case, we obtain a family of series of $G$, which we call standard series, consisting of normal subgroups and in which the index between two consecutive terms is always $p$.
To this purpose, we follow \cite{J}, but we give a construction that is a bit more general and does not require the use of elements.
Before giving the definition of these series, we remark the following immediate consequence of
(i) of \cref{powers and comms mod Pn+2} and of \cref{isomorphism free case}.

\begin{lem}
\label{index does not grow}
Let $G$ be a group, and let $L$ and $N$ be normal subgroups of $G$ such that
\[
P_n(G)\le N \le L \le P_{n-1}(G)
\vspace{5pt}
\]
for some $n$.
Then $|L^{p^m}P_{n+m+1}(G):N^{p^m}P_{n+m+1}(G)|\le |L:N|$, and if $G$ is a free group then the equality holds.
\end{lem}

The idea of the definition of standard series is very simple.
We start from an arbitrary refinement of the inclusion $P_2(G)\le G$ via a chain of normal subgroups of $G$ of maximum length.
By taking $p$th powers and multiplying by $P_3(G)$, which amounts to applying $\pi_{1,1}$, these subgroups give (after deleting repetitions) a chain of the desired type from $P_3(G)$ to $G^pP_3(G)$, that we extend to $P_2(G)$.
By iterating this process, we can refine every step from $P_{n+1}(G)$ to $P_n(G)$ with normal subgroups of $G$ whose consecutive indices are all $p$.
We formalise this idea in the following definition.

\begin{dfn}
Let $G$ be a group and let $p$ be an odd prime.
For every $n\in\N$, we define a series $\mathcal{S}(n)$ of normal subgroups of $G$,
\[
\mathcal{S}(n): P_{n+1}(G)=S_{n,0}\le S_{n,1}\le \cdots \le S_{n,s_G(n)}=P_n(G),
\]
with $|S_{n,k}:S_{n,k-1}|=p$ for $k=1,\ldots,s_G(n)$, as follows:
\begin{enumerate}
\item
The series $\mathcal{S}(1)$ is chosen arbitrarily so as to satisfy the required conditions.
\item
For $n\ge 2$, we consider the series starting at $P_{n+1}(G)$ which consists of the subgroups
\[
\{ L^p P_{n+1}(G) \mid L\in\mathcal{S}(n-1) \},
\]
and we extend it to $P_n(G)$ with normal subgroups of $G$, each of index $p$ in the next one.
\end{enumerate}
Then we say that the series $\mathcal{S}$ that is obtained by concatenating all series $\mathcal{S}(n)$ is a \emph{standard $p$-series\/} of $G$.
\end{dfn}

\begin{pro}
\label{standard series of free}
Let $F$ be a free group of rank $d$ and let $p$ be an odd prime.
If $\mathcal{S}$ is a standard $p$-series of $F$ then, for all $n\in\N$:
\begin{enumerate}
\item
If $n\ge \ell$ and $0\le k\le s_d(\ell)$ then $S_{n,k}=S_{\ell,k}^{p^{n-\ell}}P_{n+1}(F)$.
\item
If $n\ge \ell$ then $S_{n,s_d(\ell)}=Q_{n,\ell}(F)$.
\item
If $k\ge d$ then $S_{n,k+1}^p\le S_{n,k}^p[S_{n,k},F]$.
\end{enumerate}
\end{pro}

\begin{proof}
(i)
By \cref{index does not grow}, we have
\[
|S_{\ell,k}^{p^{n-\ell}}P_{n+1}(F):P_{n+1}(F)| = |S_{\ell,k}:P_{\ell+1}(F)| = p^k.
\]
Consequently $S_{\ell,k}^{p^{n-\ell}}P_{n+1}(F)=S_{n,k}$.

(ii)
If $n=\ell$ then $S_{\ell,s_d(\ell)}=P_{\ell}(F)=Q_{\ell,\ell}(F)$.
In the general case, by (i) we have
\begin{align*}
S_{n,s_d(\ell)}
&= S_{\ell,s_d(\ell)}^{p^{n-\ell}}P_{n+1}(F) = P_{\ell}(F)^{p^{n-\ell}}P_{n+1}(F)
\\
&=
\Big( \prod_{i=1}^{\ell} \, \gamma_i(G)^{p^{\ell-i}} \Big)^{p^{n-\ell}} \, P_{n+1}(F)
\\
&=
\Big( \prod_{i=1}^{\ell} \, \gamma_i(G)^{p^{n-i}} \Big) \, P_{n+1}(F)
=
Q_{n,\ell}(F),
\end{align*}
where the second to last equality follows from \cref{powers and comms mod Pn+2}.

(iii)
Since $k\ge d=s_d(1)$, we can consider the minimum $\ell$ such that $s_d(\ell-1)\le k<s_d(\ell)$.
Then
\[
S_{\ell,k+1}^p \le P_{\ell}(F)^p \le Q_{\ell+1,\ell}(F) = Q_{\ell,\ell-1}(F)^p [Q_{\ell,\ell-1}(F),F],
\]
by using \cref{special subgroups}.
On the other hand,
\[
Q_{\ell,\ell-1}(F) = S_{\ell,s_d(\ell-1)} \le S_{\ell,k} 
\]
by using (ii) and that $k\ge s_d(\ell-1)$.
Hence
\[
S_{\ell,k+1}^p \le S_{\ell,k}^p [S_{\ell,k},F].
\]
Now, by (i) and \cref{powers and comms mod Pn+2},
\begin{align*}
S_{n,k+1}^p
&=
(S_{\ell,k+1}^{p^{n-\ell}} P_{n+1}(F))^p
\le
(S_{\ell,k+1}^p)^{p^{n-\ell}} P_{n+2}(F)
\\
&\le
(S_{\ell,k}^p [S_{\ell,k},F])^{p^{n-\ell}} P_{n+2}(F)
=
(S_{\ell,k}^{p^{n-\ell}})^p [S_{\ell,k}^{p^{n-\ell}},F] P_{n+2}(F)
\\
&=
S_{n,k}^p [S_{n,k},F] P_{n+2}(F).
\end{align*}
Now, since $P_{n+2}(F)=P_{n+1}(F)^p [P_{n+1}(F),F] \le S_{n,k}^p [S_{n,k},F]$, we conclude
that $S_{n,k+1}^p\le S_{n,k}^p [S_{n,k},F]$, as desired.
\end{proof}

\section{Finding Beauville and non-Beauville $p$-groups}
\label{sec:stdbeau}

In the remainder, $p$ stands for a prime number greater than or equal to $5$.
Let $F=\langle x, y\rangle$ be the free group of rank $2$, and for an arbitrary non-negative integer $a$,
let us consider the subgroup
\[
H=\langle x^{p^a}, y \rangle^F= \langle x^{p^a}, y, y^{x}, \dots, y^{x^{p^a-1}} \rangle.
\]
Then $F/H$ is a cyclic group of order $p^a$ and $H$ is a free group in the generators given above, so free of rank $p^a+1$, by Proposition 18 in \cite{FGJ}.
Even if not explicitly mentioned in the statements of the results, $F$ and $H$ will denote these groups in the remainder of the paper.

Let us write $U_n$ for $Q_{n,2}(H)$.
In other words,
\[
U_n=H^{p^{n-1}}\gamma_2(H)^{p^{n-2}}P_{n+1}(H)
\]
for every $n\ge 2$.
Thus $U_2=P_2(H)$ and $P_{n+1}(H)<U_n<P_n(H)$ for $n\ge 3$.

The main result of this section is the determination of when a quotient $F/W$ is a Beauville group for all normal subgroups $W$ of $F$ lying between $U_{n+1}$ and $P_n(H)$.
We start by considering the power structure of quotients $F/N$ with $U_n\leq N \leq P_n(H)$.

\begin{lem}
\label{orders of generators}
Suppose that $n\ge 2$ and $U_n\leq N \leq P_n(H)$.
Then:
\begin{enumerate}
\item
$y$ has order $p^{n-1}$ modulo $N$.
\item
$xy^j$ has order $p^{a+n-1}$ modulo $N$ for all $j\in\Z$.
\end{enumerate}
\end{lem}

\begin{proof}
Since $H$ is a free group, it follows from \cref{linearly independence} that $x$ and $y$ have orders $p^{a+n-1}$ and $p^{n-1}$ modulo $N$, respectively.
Now consider the element $xy^{j}$.
We have
\[
(xy^j)^{p^a}=x^{p^a}(y^{x^{p^a-1}})^{j}(y^{x^{p^a-2}})^ j\dots (y^x)^{j}y^j \in H.
\]
Since $U_n \leq N$, by applying \eqref{pn-2 powers-1}, we get
\begin{equation}
\label{pre-key congruence}
(xy^j)^{p^{a+n-2}}
\equiv
x^{p^{a+n-2}} (y^{x^{p^a-1}})^{jp^{n-2}} \dots (y^{x})^{jp^{n-2}}y^{jp^{n-2}}
\pmod {N}.
\end{equation}
Now, by \cref{linearly independence}, all powers
\[ 
x^{p^{a+n-2}}, \ ( y^{x^{p^a-1}})^{p^{n-2}},  \dots,  (y^{x})^{p^{n-2}}, y^{p^{n-2}}
\]
are linearly independent modulo $N$, and this implies that $(xy^j)^{p^{a+n-2}} \not \in N$.
On the other hand, since
\[
\exp F/N \le \exp F/H \cdot \exp H/N \le p^{a+n-1},
\]
we have $(xy^j)^{p^{a+n-1}} \in N$.
We conclude that the order of $xy^j$ modulo $N$ is $p^{a+n-1}$.
\end{proof}

Hence, for $n\ge 2$, the minimal subgroups of $\langle y \rangle$ and $\langle xy^j \rangle$ modulo
$N$ correspond to $\langle y^{p^{n-2}} \rangle$ and $\langle (xy^j)^{p^{a+n-2}} \rangle$, respectively.
We next study the behaviour of these subgroups.

\begin{lem}
\label{conjugation by x}
Suppose that $n\ge 2$ and $U_n\leq N \leq P_n(H)$.
Then:
\vspace{3pt}
\begin{enumerate}
\setlength\itemsep{7pt}
\item
The subgroups $\langle (xy^j)^{p^{a+n-2}} \rangle$ and $\langle (xy^k)^{p^{a+n-2}} \rangle$ are different modulo $N$ if $j\not\equiv k \pmod p$, and they are also all different from
$\langle y^{p^{n-2}} \rangle$.
\item
The subgroup $\langle (xy^j)^{p^{a+n-2}} \rangle$ is central in $F$ modulo $N$ for all $j\in\Z$.
\end{enumerate}
\end{lem}

\begin{proof}
First of all, observe that, since $x^{p^a},y,y^x,\ldots,y^{x^{p^a-1}}$ are free generators of $H$,
their $p^{n-2}$nd powers are linearly independent modulo $N$ by \cref{linearly independence}.
Hence (i) follows immediately from \eqref{pre-key congruence}.

Let us now prove (ii).
On the one hand, observe that all $p^{n-2}$nd powers of elements of $H$ are central in $H$ modulo $N$, since 
\[
[H^{p^{n-2}},H]\le [Q_{n-1,1}(H),H]\le U_n \le N
\]
by \cref{special subgroups}.
Thus we only have to see that $(xy^j)^{p^{a+n-2}}$ commutes with $x$ modulo $N$.
Now, again by \eqref{pre-key congruence}, we have
\begin{equation}
\label{key congruence}
(xy^j)^{p^{a+n-2}}
\equiv
x^{p^{a+n-2}} \ \Big( (y^{x^{p^a-1}})^{p^{n-2}} \dots (y^{x})^{p^{n-2}}y^{p^{n-2}} \Big)^j
\pmod {N},
\end{equation}
and this element is invariant under conjugation by $x$, since
\[
(y^{x^{p^a}})^{p^{n-2}}  = (y^{p^{n-2}})^{x^{p^a}} \equiv y^{p^{n-2}} \pmod{N}.
\]
\end{proof}

At this point, we can already prove that all quotients $F/N$ with $U_n\le N \le P_n(H)$ are Beauville groups.

\begin{pro}
\label{F/N Beauville}
Let $N$ be a normal subgroup of $F$ such that $U_n \le N \le P_n(H)$, where $n\ge 2$.
Then the factor group $F/N$ is a Beauville group.
\end{pro}

\begin{proof}
Set $T_1=(x^{-2}, xy, x^{-1}y)$ and $T_2=(xy^2, y, xy^3)$.
We claim that the images $\overline{T}_1$, $\overline{T}_2$ of these tuples in $\overline F=F/N$ yield a Beauville structure.
Thus we have to prove that
$\langle \overline t_1^{\,\overline g} \rangle \cap \langle \overline t_2 \rangle = \overline 1$ for all
$t_1\in T_1$, $t_2\in T_2$, and $g\in F$.

Observe that $x^{-1}y$ is the inverse of $(xy^{-1})^y$.
Then for every $t_1\in T_1$, the order of $\overline t_1$ is $p^{a+n-1}$ by \cref{orders of generators}, and $\overline t_1^{\,p^{a+n-2}}$ is central by \cref{conjugation by x}.
As a consequence, the minimal subgroup of $\langle \overline t_1^{\,\overline g} \rangle$ is
$\langle \overline t_1^{\,p^{a+n-2}} \rangle$ for every $t_1\in T_1$.
On the other hand, if $t_2\in T_2$ then the minimal subgroup of $\langle \overline t_2 \rangle$ is either
$\langle \overline t_2^{\,p^{a+n-2}} \rangle$ if $t_2\ne y$, or $\langle \overline y^{\,p^{n-2}} \rangle$ otherwise. 
Now, by \cref{conjugation by x} and taking into account that $p\ge 5$, the minimal subgroups of
$\langle \overline t_1^{\,\overline g} \rangle$ and $\langle \overline t_2 \rangle$ do not coincide.
This proves that $\langle \overline t_1^{\,\overline g} \rangle \cap \langle \overline t_2 \rangle = 1$, as desired.
\end{proof}

Determining more generally which factor groups $F/W$ with $U_{n+1}\le W\le P_n(H)$ are Beauville will require a much more detailed analysis.
A fundamental role will be played by the subgroup
\[
V_n
=
\langle x^{p^{a+n-1}}, (y^{x^{p^a-1}} \ldots y^x y)^{p^{n-1}} \rangle \, U_{n+1}.
\]
Note that by \eqref{key congruence}, and taking into account that $p^{n-1}$st powers of elements of $H$ are central in $H$ modulo $U_{n+1}$, we get
\[
(xy^j)^{p^{a+n-1}} 
\equiv
x^{p^{a+n-1}} \big( y^{x^{p^a-1}} \ldots y^x y \big)^{jp^{n-1}}
\pmod {U_{n+1}}.
\]
Thus, we also have
\[
V_n
=
\langle (xy^j)^{p^{a+n-1}} \mid j\in\Z \rangle \, U_{n+1}.
\]
Observe that $U_{n+1}\le V_n\le U_n$ and that, by \cref{conjugation by x}(ii),
$V_n/U_{n+1}\le  Z(F/U_{n+1})$.
In the next lemma we give more precise information about the relation of $V_n$ with the centre of
$F/U_{n+1}$.
For simplicity, we will write $z$ for $y^{x^{p^a-1}} \ldots y^x y$ in the remainder of this section.

\begin{pro}
\label{Vn}
For every $n\ge 2$, we have:
\begin{enumerate}
\item
$V_n/U_{n+1}=H^{p^{n-1}}U_{n+1}/U_{n+1}\cap Z(F/U_{n+1})$.
\item
$V_n/U_{n+1}$ is elementary abelian of order $p^2$.
\end{enumerate}
\end{pro}

\begin{proof}
By \eqref{pn-2 powers-1}, we have
\[
H^{p^{n-1}}U_{n+1}
=
\langle x^{p^{a+n-1}}, (y^{p^{n-1}})^{x^j} \mid j=0,\ldots,p^a-1 \rangle \, U_{n+1}.
\]
On the other hand, \cref{powers and comms mod Pn+2} yields
\[
(H^{p^{n-1}})^p\le H^{p^n} P_{n+2}(H)\le U_{n+1}
\]
and
\[
[H^{p^{n-1}},H]\le \gamma_2(H)^{p^{n-1}} P_{n+2}(H)\le U_{n+1}.
\]
It follows that $H^{p^{n-1}}U_{n+1}/U_{n+1}$ is an elementary abelian $p$-group on which $H$ acts trivially by conjugation.
As a consequence, if we set $C=F/H=\langle xH \rangle \cong C_{p^a}$ then
$H^{p^{n-1}}U_{n+1}/U_{n+1}$ is an $\F_pC$-module.

Now, by \cref{linearly independence}, the elements $x^{p^{a+n-1}}$ and $(y^{p^{n-1}})^{x^j}$, for $j=0,\ldots,p^a-1$, are linearly independent in $H^{p^{n-1}}U_{n+1}/U_{n+1}$.
Thus we have the following decomposition as a direct sum of $\F_pC$-modules:
\begin{equation}
\label{direct sum}
H^{p^{n-1}}U_{n+1}/U_{n+1}
=
\langle x^{p^{a+n-1}}U_{n+1} \rangle \oplus  \langle (y^{p^{n-1}}U_{n+1})^{x^j} \mid j=0,\ldots,p^a-1 \rangle,
\end{equation}
where the first direct summand is a trivial $\F_pC$-module and the second is free of rank $1$, so isomorphic to $\F_pC$.

Observe that $H^{p^{n-1}}U_{n+1}/U_{n+1}\cap Z(F/U_{n+1})$ is nothing but the maximal trivial
$\F_pC$-submodule of $H^{p^{n-1}}U_{n+1}/U_{n+1}$.
According to Proposition 6.1.1 and Theorem 6.1.2 of \cite{We}, the group algebra $\F_pC$ has a unique composition series, and as a consequence its maximal trivial submodule is of $\F_p$-dimension $1$.
Then by \eqref{direct sum}, $H^{p^{n-1}}U_{n+1}/U_{n+1}\cap Z(F/U_{n+1})$ is of order $p^2$, and since 
\[
z^{p^{n-1}} U_{n+1} \in \langle (y^{p^{n-1}}U_{n+1})^{x^j} \mid j=0,\ldots,p^a-1 \rangle
\]
is non-trivial and fixed under the action of $x$, this intersection coincides with $V_n/U_{n+1}$, as desired.
\end{proof}

The factor groups $F/P_2(F)$ and $V_n/U_{n+1}$ are both elementary abelian of order $p^2$, generated by $xP_2(F)$ and $yP_2(F)$, and by $x^{p^{a+n-1}}U_{n+1}$ and $z^{p^{n-1}}U_{n+1}$, respectively.
In our next proposition we see that taking $p^{a+n-1}$st powers gives a bijection between $p$ of the $p+1$  subgroups of $F/P_2(F)$ and of $V_n/U_{n+1}$ of order $p$.
We need a couple of lemmas.

\begin{lem}
\label{conjugates of elements in F'}
If $w \in F'$ then 
\[
w^{x^{p^a-1}}\dots w^{x}w \in P_2(H).
\]
\end{lem}

\begin{proof}
Since $F' \leq H$ and $H$ is abelian modulo $P_2(H)$, we have
\[
F' = \langle [x, y]^{x^j} \mid 0\leq j\leq p^a-1 \rangle P_2(H).
\]
Thus, in order to prove the lemma, it is enough to consider the case $w=[x, y]$.
Then
\[
w^{x^{p^a-1}}\dots w^{x}w
=
x^{-p^a} (xw)^{p^a}
=
x^{-p^a} (x^y)^{p^a}
=
[x^{p^a}, y] \in P_2(H).
\]
\end{proof}

\begin{lem}
\label{pa-th powers}
If $g\equiv x^iy^j \pmod{P_2(F)}$ and $p$ does not divide $i$, then
\begin{equation}
\label{congruence pa-th powers}
g^{p^a}\equiv (x^{p^a})^i z^j \pmod{P_2(H)}.
\end{equation}
\end{lem}

\begin{proof}
Since $P_2(F)=\langle x^p,y^p \rangle F'$, we can write $g=x^k y^{\ell} w$, with $k\equiv i\pmod p$,
$\ell\equiv j\pmod p$ and $w\in F'$.
Set $h=y^{\ell} w\in H$.
Then
\[
g^{p^a}
=
(x^{p^a})^k h^{x^{k(p^a-1)}} \ldots h^{x^k} h.
\]
Since $p$ does not divide $i$, the subgroups $\langle x^k \rangle$ and $\langle x \rangle$ coincide modulo $H$, and since $H$ is abelian modulo $P_2(H)$, it follows that
\begin{align*}
g^{p^a}
&\equiv
(x^{p^a})^k \, h^{x^{p^a-1}} \ldots h^x h
\\
&\equiv
(x^{p^a})^k \, z^{\ell} \, (w^{x^{p^a-1}} \ldots w^x w)
\\
&\equiv
(x^{p^a})^k \, z^{\ell}
\pmod{P_2(H)},
\end{align*}
by using \cref{conjugates of elements in F'}.
Now, since $H/P_2(H)$ is of exponent $p$, we obtain the congruence in
\eqref{congruence pa-th powers}.
\end{proof}

\begin{pro}
\label{bijection}
Let $n\ge 2$ and let $\mathcal M$ and $\mathcal N$ be the sets of all subgroups of $F/P_2(F)$ and $V_n/U_{n+1}$ of order $p$ other than $\langle yP_2(F) \rangle$ and
$\langle z^{p^{n-1}} U_{n+1} \rangle$, respectively.
Then the map
\[
\begin{matrix}
\Phi & \colon & \mathcal{M} & \longrightarrow & \mathcal{N}
\\
& & \langle gP_2(F) \rangle & \longmapsto & \langle g^{p^{a+n-1}} U_{n+1} \rangle
\end{matrix}
\]
is a bijection.
\end{pro}

\begin{proof}
First of all, let us see that $\Phi$ is well defined, i.e.\ that it does not depend on the choice of the generator $gP_2(F)$.
Since $\langle gP_2(F) \rangle$ belongs to $\mathcal M$, it is equal to $\langle xy^kP_2(F) \rangle$ for some $k\in\{0,\ldots,p-1\}$.
Then $g\equiv x^iy^{ki}\pmod{P_2(F)}$ for some $i$ coprime to $p$.
By combining \cref{pa-th powers} and \cref{pn-2 powers}, we get
\[
g^{p^{a+n-1}}
\equiv
(x^{p^{a+n-1}})^i (z^{p^{n-1}})^{ki}
\equiv
(x^{p^{a+n-1}} \, z^{kp^{n-1}})^i \pmod{U_{n+1}}.
\]
Similarly,
\begin{equation}
\label{Phi surjective}
(xy^k)^{p^{a+n-1}} \equiv x^{p^{a+n-1}} \, z^{kp^{n-1}} \pmod{U_{n+1}},
\end{equation}
and consequently $\langle g^{p^{a+n-1}} U_{n+1} \rangle = \langle (xy^k)^{p^{a+n-1}} U_{n+1} \rangle$.
This proves that $\Phi$ is well defined.

Finally, \eqref{Phi surjective} shows that $\Phi$ is surjective and, since $|\mathcal{M}|=|\mathcal{N}|=p$,
we conclude that it is bijective.
\end{proof}

We need one more ingredient before proceeding to the proof of the main result of this section.
Beauville structures are not necessarily inherited when passing to a quotient and cannot be in general lifted from a quotient.
However, as we see in the next two lemmas, we can preserve Beauville structures in some special cases.

\begin{lem}{\cite[Lemma 4.2]{FJ}}
\label{lifting structure}
Let $G$ be a finite $2$-generator group, and let $T_1$ and $T_2$ be two spherical triples of generators of $G$.
Assume that $N\trianglelefteq G$ is such that:
\begin{enumerate}
\item
$(\overline T_1,\overline T_2)$ is a Beauville structure for $\overline G=G/N$.
\item
$o(t_1)=o(\overline t_1)$ for every $t_1\in T_1$.
\end{enumerate}
Then $(T_1,T_2)$ is a Beauville structure for $G$.
\end{lem}

\begin{lem}
\label{reducing BS}
Let $G$ be a finite $p$-group with a Beauville structure $(T_1,T_2)$, and let $N\trianglelefteq G$.
Suppose that there exists $K\trianglelefteq G$ such that:
\begin{enumerate}
\item
For every $t\in T_1\cup T_2$, the subgroup of order $p$ of $\langle t \rangle$ lies in $K$.
\item
$K\cap N=1$.
\end{enumerate}
Then $(\overline T_1,\overline T_2)$ is a Beauville structure for $\overline G=G/N$.
\end{lem}

\begin{proof}
First of all, (ii) implies that $o(\overline t)=o(t)$ for all $t\in T_1\cup T_2$.
By way of contradiction, assume that there exist $t_1\in T_1$, $t_2\in T_2$ and $g\in G$ such that
$\langle \overline t_1^{\,\overline g} \rangle \cap \langle \overline t_2 \rangle \ne \overline 1$.
Then the subgroups of order $p$ of these two cyclic subgroups coincide.
Thus if we set $n_1=o(t_1)/p$ and $n_2=o(t_2)/p$, we have
$\big(\overline t_1^{\,n_1} \big)^{\overline g}=\overline t_2^{\,in_2}$ for some $i$ coprime to $p$.
By (i), and since $K\trianglelefteq G$, we get
\[
\big( t_1^{n_1} \big)^g \, t_2^{-in_2} \in K\cap N=1.
\]
Hence $\big( t_1^{n_1} \big)^g=t_2^{in_2}$.
This is a contradiction, since $(T_1,T_2)$ is a Beauville structure for $G$.
\end{proof}

We can finally present the main result of this section.

\begin{thm}
\label{thm:beauville iff}
Let $W$ be a normal subgroup of $F$ such that $U_{n+1}\le W\le P_n(H)$, with $n\ge 2$.
Then $F/W$ is a Beauville group if and only if $W\cap V_n$ is either equal to $V_n$ or to $U_{n+1}$.
\end{thm}

\begin{proof}
We first assume that $W\cap V_n=V_n$, i.e.\ that $V_n\le W$.
By the proof of \cref{F/N Beauville}, $F/P_n(H)$ has a Beauville structure whose first triple is given by the images of $T_1=(x^{-2}, xy, x^{-1}y)$.
Let $t_1\in T_1$.
By \cref{orders of generators}, the order of $t_1$ modulo $P_n(H)$ is $p^{a+n-1}$.
On the other hand, by \eqref{key congruence}, $t_1^{p^{a+n-1}}$ lies in $V_n$, and so also in $W$.
Thus $t_1$ has the same order modulo $P_n(H)$ and modulo $W$ and, by \cref{lifting structure}, the Beauville structure for $F/P_n(H)$ can be lifted to a Beauville structure for $F/W$.

Suppose now that $W\cap V_n=U_{n+1}$.
We already know a Beauville structure $(\overline T_1,\overline T_2)$ for $\overline F=F/U_{n+1}$,  by \cref{F/N Beauville}.
We are going to prove that $F/W$ is a Beauville group by applying \cref{reducing BS} to $G=F/U_{n+1}$, with $N=W/U_{n+1}$ and $K=H^{p^{n-1}}U_{n+1}/U_{n+1}$.
If we consider the elements of $T_1$ and $T_2$ in the proof of \cref{F/N Beauville} and we use
\cref{orders of generators}, it is clear that condition (i) of \cref{reducing BS} holds.
Thus we only need to see that $K\cap N$ is trivial.
If not, then this non-trivial normal subgroup of the finite $p$-group $F/U_{n+1}$ must intersect non-trivially the centre, that is, we have
\[
\frac{W}{U_{n+1}} \cap \frac{H^{p^{n-1}}U_{n+1}}{U_{n+1}} \cap Z\Big( \frac{F}{U_{n+1}} \Big)
\ne \overline 1.
\]
By (i) of \cref{Vn}, it follows that $W/U_{n+1}\cap V_n/U_{n+1} \ne \overline 1$, which is impossible since
$W\cap V_n=U_{n+1}$.

Finally, we assume that $W\cap V_n$ is different from $V_n$ and from $U_{n+1}$, and we prove that
$\overline F=F/W$ is not a Beauville group.
By \cref{Vn}, we have $|V_n:U_{n+1}|=p^2$, and consequently
$|V_n:W\cap V_n|=|W\cap V_n:U_{n+1}|=p$.
Let us write $\mathcal{M}^*$ for the set of subgroups of $\mathcal{M}$ whose image under the map $\Phi$ of \cref{bijection} is different from $(W\cap V_n)/U_{n+1}$.
If $\mathfrak{M}^*$ is the set of maximal subgroups of $F$ containing $P_2(F)$ that correspond to the subgroups in $\mathcal{M}^*$, then $|\mathfrak{M}^*|\ge p-1$ by \cref{bijection}.
From the definition of $\Phi$, for all $M\in\mathfrak{M}^*$ and all $g\in M\smallsetminus P_2(F)$
we have $\langle g^{p^{a+n-1}} \rangle U_{n+1} \ne W\cap V_n$, and consequently
$\langle g^{p^{a+n-1}} \rangle (W\cap V_n)=V_n$.
Hence for all such elements $g$, we have
\begin{equation}
\label{collision}
\langle \overline g^{\,p^{a+n-1}} \rangle=V_nW/W
\ne \overline 1.
\end{equation}

Now assume, by way of contradiction, that $\overline F$ has a Beauville structure
$(\overline T_1,\overline T_2)$.
Then the elements of $\overline T_1$ lie in three different maximal subgroups of $\overline F$, and similarly for $\overline T_2$.
Since $\overline F$ has $p+1$ maximal subgroups and $|\mathfrak{M}^*|\ge p-1$, there exist
$t_1\in T_1$ and $t_2\in T_2$ with $t_1\in M_1\smallsetminus P_2(F)$ and
$t_2\in M_2\smallsetminus P_2(F)$ for some $M_1,M_2\in \mathfrak{M}^*$.
Then, by \eqref{collision},
\[
\langle \overline t_1^{\,p^{a+n-1}} \rangle=V_nW/W=\langle \overline t_2^{\,p^{a+n-1}} \rangle
\]
coincide, which is impossible if $(\overline T_1,\overline T_2)$ is a Beauville structure.
\end{proof}

\section{Asymptotics}

In this section we combine the results of \cref{sec:stdbeau} with the counting method of \cite{J} in order to obtain asymptotic lower bounds for both $\mathrm{b}(p^t)$ and $\mathrm{nb}(p^t)$.
We need to adapt the arguments of \cite{J} to our specific situation, which requires some preliminary work.
Also, even if most of the time we simply quote the results in \cite{J}, on occasions we have partially rewritten some of the proofs for the convenience of the reader.

In the remainder, $p$ denotes a fixed prime with $p\ge 5$.
Also, as in \cref{sec:stdbeau}, $F$ is the free group on two generators $x$ and $y$, and
$H=\langle x^{p^a},y \rangle^F$, for some non-negative integer $a$.
Then $H$ is of index $p^a$ in $F$, and is a free group in the $p^a+1$ generators
$\{x^{p^a},y^{x^i}\mid i=0,\ldots,p^a-1\}$.

We consider a standard series $\mathcal S=\{S_{n,k}\}$ of $H$, constructed in such a way that, every time we need to introduce subgroups which are not $p$th powers of previous ones, we choose them to be normal in $F$, not only in $H$.
Then $\mathcal S$ consists entirely of normal subgroups of $F$.
The Beauville and non-Beauville groups that we will be counting will all arise in the form
$F/W$ with $S_{n,k}^p[S_{n,k},H]\le W\le S_{n,k}$ for some $n$ and $k$.
We will call these factor groups \emph{standard $2$-generator $p$-groups\/}.
This makes it necessary to study the quotients $S_{n,k}/S_{n,k}^p[S_{n,k},H]$.

\begin{rmk}
For every subgroup $T$ of $H$ which is normal in $F$, the section $T/T^p[T,H]$ is an
$\F_p[C_{p^a}]$-module, since $H$ acts trivially on it.
Then we denote by $\phantom{i}^\ast : T\to T/T^p[T,H]$ the natural projection.
\end{rmk}

Our goal is to find a direct sum decomposition for the module $S_{n,k}^*$ that will allow us to count easily standard groups $F/W$ for which we know  \textit{a priori} the intersection $W\cap V_n$ and so, by  \cref{thm:beauville iff}, whether $F/W$ is a Beauville group or not.
We do this in the following lemmas.

\begin{lem}
\label{order of Qn1}
For every $n\in\N$, we have $|Q_{n,1}(H):P_{n+1}(H)|=p^{\,p^a+1}$.
As a consequence, $Q_{n,1}(H)=S_{n,p^a+1}$.
\end{lem}

\begin{proof}
The map $\pi_{1,n-1}:H/P_2(H)\rightarrow P_n(H)/P_{n+1}(H)$ of
\cref{isomorphism free case} is a monomorphism whose image is
\[
H^{p^{n-1}}P_{n+1}(H)/P_{n+1}(H)=Q_{n,1}(H)/P_{n+1}(H).
\]
Since $|H:P_2(H)|=p^{\,p^a+1}$, the first assertion follows.
As for the second, observe that any standard series $\{S_{n,k}\}$ refines the series $\{Q_{n,k}(H)\}$,
and that $|S_{n,k}:P_{n+1}(H)|=p^k$.
\end{proof}

\begin{lem}
\label{lem:split}
Let $T$ be a normal subgroup of $F$ such that $Q_{n,1}(H) \le T \le P_{n}(H)$.
Then $(H^{p^{n-1}})^*$ is a direct summand of $T^*$ as an $\F_p[C_{p^a}]$-module.
\end{lem}

\begin{proof}
By induction on $n$.
For $n=1$, the assumption implies that $T=H$ and the result trivially holds. 

Suppose now that $n\ge 2$.
Let us denote by $\overline{\phantom{i}}: F\to F/P_{n+1}(H)$ the natural projection.
We first prove that $\overline{H^{p^{n-1}}}$, or what is the same, $\overline{Q_{n,1}(H)}$, admits a complement in $\overline T$.
Consider the map $f_n:\mathfrak{D}_n(H)\rightarrow \overline{P_n(H)}$ of \cref{huppert bijection} corresponding to the free group $H$, which is clearly an $\F_p[C_{p^a}]$-isomorphism.
Since
\[
f_n(H/H^p\gamma_2(H) \times 1 \times \ldots \times 1)
=
\overline{H^{p^{n-1}}}
\]
and $H/H^p\gamma_2(H) \times 1 \times \ldots \times 1$ obviously admits a complement in
$\mathfrak{D}_n(H)$, it follows that $\overline{H^{p^{n-1}}}$ also admits a complement in
$\overline{P_n(H)}$.
By intersecting with $T$, we obtain a normal subgroup $R$ of $F$ such that
$\overline T=\overline{H^{p^{n-1}}} \oplus \overline R$, as desired.

Now we prove that $(H^{p^{n-1}})^*$ admits a complement in $T^*$.
Since $T\le P_n(H)$, we have $T^p[T,H]\le P_{n+1}(H)$ and there is a quotient map from
$T^*$ to $\overline T$.
If we show that the restriction $(H^{p^{n-1}})^* \rightarrow \overline{H^{p^{n-1}}}$ of this map is an isomorphism, then
\[
H^{p^{n-1}} \cap P_{n+1}(H) = H^{p^{n-1}} \cap T^p[T,H],
\]
and it readily follows that $R^*$ is a complement of $(H^{p^{n-1}})^*$ in $T^*$.
Consider the subgroup
\begin{equation}
\label{Jn}
J_n=\langle x^{p^{a+n-1}}, (y^{p^{n-1}})^{x^j} \mid j=0,\ldots,p^a-1 \rangle.
\end{equation}
By \cref{powers and comms mod Pn+2}, we have $\overline{H^{p^{n-1}}}=\overline{J_n}$.
On the other hand, since $Q_{n,1}(H)\le T$, we have
\[
U_{n+1} = Q_{n+1,2}(H) = Q_{n,1}(H)^p [Q_{n,1}(H),H] \le T^p[T,H],
\]
and so also $(H^{p^{n-1}})^*=J_n^*$, this time by \cref{pn-2 powers}.
Now observe that $J_n^*$ has order at most $p^{\,p^a+1}$, since it can be generated by $p^a+1$ elements.
Since
\begin{equation}
\label{size forbidden}
|\overline{J_n}|=|\overline{Q_{n,1}(H)}|=p^{\,p^a+1}
\end{equation}
by \cref{order of Qn1}, it follows that the natural map from $J_n^*$ to $\overline J_n$ is an isomorphism.
This completes the proof.
\end{proof}

By \cref{order of Qn1}, the subgroups $S_{n,k}$ satisfy the hypotheses of the previous result for all $n$ and for $k\ge p^a+1$.
In our next lemma we follow the steps of \cite[Corollary 3.5 and Claim~1]{J} in order to explore further the structure of the section $S_{n,k}^{\ast}=S_{n,k}/S_{n,k}^p[S_{n,k},H]$.

\begin{lem}
\label{Snk star}
For every $n\in\N$, the following holds:
\begin{enumerate}
\item
For all $k\in\{0,\ldots,s_d(n)\}$,
\[
\log_p |S_{n,k}^*| \ge p^a \log_p |F:S_{n,k}| (1+o(1)),
\]
as $n$ goes to infinity.
\item
If $k\ge p^a+1$ then $S_{n,k}^{\ast}$ can be written as a direct sum of $\F_p[C_{p^a}]$-modules,
\[
(H^{p^{n-1}})^*\oplus A_{n,k}\oplus B_{n,k},
\]
where $B_{n,k}$ is free with
\[
\rank B_{n,k} \ge p^a\log_p |F:S_{n,k}|(1+o(1))
\]
as $n$ goes to infinity.
\end{enumerate}
\end{lem}

\begin{proof}
(i)
Set $d=p^a$.
By \cref{standard series of free}, $S_{n,k+1}^p\le S_{n,k}^p[S_{n,k},F]$ for $k\ge d$.
Arguing as in \cite[Lemma 2.1]{J}, we get $|S_{n,k}^*|\le p^{d-1}|S_{n,k+1}^*|$ for $k\ge d$,
while $|S_{n,k}^*|\le p^d \, |S_{n,k+1}^*|$ for $k=0,\ldots,d-1$.
Since $P_{n+1}(H)=S_{n,0}$, it readily follows that
\[
|P_{n+1}(H)^*|\le p^{(d-1)k+d} \, |S_{n,k}^*|,
\quad
\text{for all $k=0,\ldots,s_d(n)$,} 
\]
i.e.\ that
\begin{equation}
\label{bound Snk star}
\log_p |S_{n,k}^*| \ge \log_p |P_{n+1}(H)^*| - (d-1)k - d.
\end{equation}
Now, by \cite[Lemma 3.2]{J}, we have
\[
\log_p |P_{n+1}(H)^*|=(d-1)\log_p |F:P_{n+1}(H)|(1+\alpha(n)),
\]
with $\alpha(n)=o(1)$ as $n$ goes to infinity.
If $\beta(n)=\min\{\alpha(n),0\}$ for all $n\in\N$ then also $\beta(n)=o(1)$ and
\[
\log_p |P_{n+1}(H)^*|\ge (d-1)\log_p |F:P_{n+1}(H)|(1+\beta(n)).
\]
If we combine this with \eqref{bound Snk star}, taking into account that $\log_p |S_{n,k}:P_{n+1}(H)|=k$, we obtain that
\[
\log_p |S_{n,k}^*|
\ge
(d-1)\log_p |F:S_{n,k}|(1+\beta(n)) + (d-1)k \beta(n) - d.
\]
Since, by \eqref{Pn/Pn+1 versus F/Pn},
\[
k \le \log_p |P_n(H):P_{n+1}(H)| \le d \log_p |H:P_n(H)| \le d \log_p |F:S_{n,k}|,
\]
and since $\beta(n)\le 0$, we get
\[
\log_p |S_{n,k}^*|
\ge
(d-1)\log_p |F:S_{n,k}|(1+(d+1)\beta(n)) - d,
\]
and the result follows.

(ii)
Since (i) holds, the proof of Claim 1 of \cite{J} applies to get a direct sum of $\F_p[C_{p^a}]$-modules
\begin{equation}
\label{free summand}
S_{n,k}^*=C_{n,k}\oplus D_{n,k},
\end{equation}
where $D_{n,k}$ is a maximal free submodule and
\[
\rank D_{n,k} \ge p^a\log_p |F:S_{n,k}|(1+o(1)).
\]
On the other hand, by \cref{lem:split}, we have $S_{n,k}^*=(H^{p^{n-1}})^*\oplus X_{n,k}$ for some
$\F_p[C_{p^a}]$-module $X_{n,k}$.
As seen in the proof of \cref{lem:split}, we have $(H^{p^{n-1}})^*=J_n^*$, where $J_n$ is given by
\eqref{Jn}, and $|(H^{p^{n-1}})^*|=p^{\,p^a+1}$.
It follows that
\[
(H^{p^{n-1}})^* = \langle (x^{p^{a+n-1}}) \rangle^* \oplus
\langle (y^{p^{n-1}})^{x^j} \mid j=0,\ldots,p^a-1 \rangle^*
\]
is a sum of a trivial $\F_p[C_{p^a}]$-module and a free $\F_p[C_{p^a}]$-module.
Now if we write $X_{n,k}$ as a sum of indecomposable modules and group together the modules which are isomorphic to $\F_p[C_{p^a}]$, we get a free submodule $B_{n,k}$ such that
$X_{n,k}=Y_{n,k}\oplus B_{n,k}$.
Hence
\[
S_{n,k}^*= \langle (x^{p^{a+n-1}}) \rangle^* \oplus
\langle (y^{p^{n-1}})^{x^j} \mid j=0,\ldots,p^a-1 \rangle^*\oplus Y_{n,k}\oplus  B_{n,k}.
\]
Comparing this with \eqref{free summand} and applying the Krull-Schmidt theorem to the finite module
$S_{n,k}^*$, it follows that $\rank B_{n,k}=\rank D_{n,k}-1$, and we are done.
\end{proof}

We also recall the following result from \cite{J}.

\begin{lem}{\cite[Claim~2]{J}}
\label{lem:countsubmod}
Let $M$ be a free $\F_p[C_{p^a}]$-module of rank $m$.
Then for every $s\in\{0,\ldots,m\}$, there are at least $p^{s(m-s)p^a}$ submodules $N$ of $M$ such that $M/N$ is a free $\F_p[C_{p^a}]$-module of rank $s$.
\end{lem}

We are finally ready for the proof of the Main Theorem.

\begin{thm} 
The number $\bea(p^t)$ of Beauville groups of order $p^t$ and the number $\nbea(p^t)$ of non-Beauville $2$-generator groups of order $p^t$ both grow at least as
$p^{t^2/4 + o(t^2)}$ as $t$ goes to infinity.
\end{thm}

\begin{proof}
Let $\mathcal{S}$ be a standard series of $H$ consisting of normal subgroups of $F$.
The idea of the proof is to consider factor groups of the form $F/W$, where $T^p[T,H]\le W\le T$ for some subgroup $T$ in the series $\mathcal{S}$.
By suitably choosing $T$, we will be able to apply \cref{thm:beauville iff} and so to determine whether the quotients $F/W$ are Beauville or not.
Then a suitable choice of the subgroups $W$, together with the previous lemmas, will ensure that we get as many groups of each type as desired.

Let us make these ideas more precise.
Set $d=p^a+1$ and, for every $n\ge 1$, let
\[
\mathcal{H}(n)
=
\{ S_{n,k} \mid k=d,\ldots,s_d(n) \}.
\]
Observe that $S_{n,0}=S_{n+1,s_d(n+1)}$ belongs to $\mathcal{H}(n+1)$.
If $T\in\mathcal{H}(n)$ then $Q_{n,1}(H)\le T\le P_n(H)$, by \cref{order of Qn1}.
Consequently, if $W$ is a normal subgroup of $F$ lying between $T^p[T,H]$ and $T$, we have
\begin{multline*}
U_{n+1}=Q_{n+1,2}(H)=Q_{n,1}(H)^p [Q_{n,1}(H),H]
\\
\le T^p[T,H] \le W \le T \le P_n(H).
\end{multline*}
As a consequence, \cref{thm:beauville iff} allows us to determine whether the quotient $F/W$ is a Beauville group or not.

Thus we are going to work with submodules of the $\F_p[C_{p^a}]$-module $T^\ast=T/T^p[T,H]$.
By \cref{Snk star}, $T^\ast=(H^{p^{n-1}})^*\oplus A\oplus B$ decomposes as a direct sum of submodules, where $B$ is free, of rank $m$, say.
Now consider a submodule $C$ of $B$ such that the quotient $B/C$ is free of rank $s$, where
$s\in\{0,\ldots,m\}$ is arbitrary for the moment.
Let $M$ and $N$ be the inverse images in $T$ of $A$ and $C$, respectively.
Then, by \cref{thm:beauville iff}, the following hold:
\begin{enumerate}
\item
If $W=H^{p^{n-1}}MN$ then $V_n\le H^{p^{n-1}}U_{n+1}\le W$, and $F/W$ is a Beauville group.
\item
If $W=\langle x^{p^{a+n-1}}\rangle MN$ then
$x^{p^{a+n-1}}\in (V_n\cap W)\smallsetminus U_{n+1}$ and
\[
(V_n)^*\cap W^* \subseteq (H^{p^{n-1}})^*\cap W^* = \langle x^{p^{a+n-1}} \rangle^* \subsetneq
(V_n)^*.
\]
This implies that $U_{n+1}\subsetneq V_n\cap W\subsetneq V_n$, and $F/W$ is not a Beauville group.
\end{enumerate}
Since $W$ as above is in any case of the form $KN$ with $K\le H^{p^{n-1}}M$ fixed and $N$ variable, and since $T^*=(H^{p^{n-1}})^*\oplus A\oplus B$, we get as many Beauville or non-Beauville groups as choices there are for $C$.

In order to obtain groups of order exactly $p^t$, we now choose $T$ and $s$ suitably.
We may assume that $t\ge 2p^a$, and then we set $r=\lfloor t/2p^a \rfloor$.

We first consider the case when we want to get Beauville groups.
Then $t-rp^a\ge p^a$ and there exist unique $n,k\in\N$ such that $|F:S_{n,k}|=p^{t-rp^a}$.
If $S_{n,k}\in\mathcal{H}(n)$ then we choose $T=S_{n,k}$.
Otherwise $k\in\{1,\ldots,d-1\}$, and then we choose $T$ to be the subgroup of the series $\mathcal{S}$ that has index $p^a$ in $S_{n,k}$.
In other words,
\[
T = S_{n+1,s_d(n+1)-d+k+1}.
\]
Observe that $T\in\mathcal{H}(n+1)$ in this case, since
\[
s_d(n+1)-d+k+1\ge s_d(2)-d+2 = \frac{d^2+d}{2}-d+2 \ge d.
\]
Thus $T$ is either in $\mathcal{H}(n)$ or $\mathcal{H}(n+1)$, and $|F:T|=p^{t-sp^a}$, where $s=r$ or $r-1$.
In any case, we have $s=t/2p^a+O(1)$.

Since in this case we set $W=H^{p^{n-1}}MN$ and
\[
|T:W|=|T^*:W^*|=|B:C|=p^{sp^a},
\]
it follows that $F/W$ is of order $p^t$.
Let $\bq(p^t)$ be the number of Beauville quotients of order $p^t$ that we are obtaining this way, i.e.\ the number of ways of choosing $C$ inside $B$ with $B/C$ free of rank $s$.
By \cref{lem:countsubmod}, we have
\[
\log_p \bq(p^t) \ge  s(m-s)p^a,
\]
and by \cref{Snk star},
\[
m \ge p^a(t-sp^a)(1+o(1)).
\]
It follows that
\begin{align*}
\log_p \bq(p^t)
&\ge 
\Big( \frac{t}{2} + O(1) \Big) \Big( \Big( \frac{t}{2} + O(1) \Big)(1+o(1)) - \frac{t}{2p^a} + O(1) \Big)
\\
&=
\Big( \frac{t}{2} + O(1) \Big) \Big( \frac{p^a-1}{2p^a}t + o(t) \Big)
\\
&=
 \frac{p^a - 1}{4p^a} t^2 +o(t^2).
\end{align*}
Since $a$ can be made as big as we want, this implies that
\begin{equation}
\label{bound bq}
\log_p \bq(p^t) \ge \frac{1}{4} t^2 +o(t^2).
\end{equation}

Now, in order to bound $\bea(p^t)$, we have to consider how many quotients $F/W$ can be isomorphic to a given one, say $P$.
This is at most the number of homomorphisms from $F$ to $P$, which since $|P|=p^t$ and $F$ has $2$ generators, is $p^{2t}$.
Thus
\[
\bea(p^t) \ge \bq(p^t)/p^{2t}
\]
and we conclude from \eqref{bound bq} that
\[
\log_p \bea(p^t) \ge \frac{1}{4} t^2 +o(t^2),
\]
as desired.

Now let us briefly explain the changes that need to be done for the counting of non-Beauville groups.
As before, we choose $S_{n,k}$ such that $|F:S_{n,k}|=p^{t-rp^a}$ and we let $T$ be either $S_{n,k}$ or the subgroup in the series $\mathcal{S}$ of index $p^a$ in $S_{n,k}$, in order to make sure that $T$ belongs to either $\mathcal{H}(n)$ or $\mathcal{H}(n+1)$.
Then we have $|F:T|=p^{t-(s+1)p^a}$ with $s=r-1$ or $r-2$.
Again, $s=t/2+O(1)$.

Recall that in order to get non-Beauville groups we set
$W=\langle x^{p^{a+n-1}} \rangle MN$.
Then
\[
|T:W|=|T^*:W^*|=|(H^{p^{n-1}})^*:\langle x^{p^{a+n-1}} \rangle^*| \cdot |B:C| =p^{(s+1)p^a},
\]
and it follows that $F/W$ is of order $p^t$.
The rest of the proof is exactly the same as above.
\end{proof}

\end{document}